\documentclass[a4paper,11pt]{amsart}
\usepackage{amsfonts,amssymb,epsfig,latexsym}
\usepackage{amsmath}
\usepackage[latin1]{inputenc}
\usepackage{color,layout}
\usepackage{ifthen}

\usepackage{pdfsync}
\usepackage{graphicx}
\usepackage{color}

\newtheorem{theorem}{Theorem}[section]
\newtheorem{lemma}[theorem]{Lemma}
\newtheorem{proposition}[theorem]{Proposition}

\newtheorem{remark}[theorem]{Remark}

\def\square{\hbox{\vrule\vbox{\hrule\phantom{o}\hrule}\vrule}}

\def\R{\mathbb {R}}

\def\S{\mathbb {S}}

\def\tendsto{\rightarrow}
\def\re{\mathop{\rm Re}\nolimits}
\def\im{\mathop{\rm Im}\nolimits}

\def\O{\mathcal O}

\def\eps{\varepsilon}
\newcommand{\indic}{1\!\!1}
\newcommand{\wbar}{\overline{w}}

\newcommand{\be}{\begin{equation}}
\newcommand{\ee}{\end{equation}}

\numberwithin{equation}{section}

\begin{document}

\title[Resonance widths for Helmholtz resonators]{Resonance widths for general Helmholtz Resonators with straight neck}
 
\begin{abstract} 
We prove an optimal exponential 
lower bound on the width of the resonance associated to the first eigenvalue of the cavity for a general Helmholtz resonator with straight neck, in any dimension. 
\end{abstract}

\author{ Thomas Duyckaerts${}^1$, Alain Grigis${}^1$ \& Andr\'e Martinez${}^2$}


\maketitle

\addtocounter{footnote}{1}
\footnotetext{Universit\'e Paris 13, Institut Galil\'ee, D\'epartement de Math\'ematiques, avenue J.-B. Cl\'ement, 93430 Villetaneuse, France. TD is partially supported by ANR grant SchEq, ERC grant Dispeq and ERC advanced grant BLOWDISOL.}
\addtocounter{footnote}{1}
\footnotetext{Universit\`a di Bologna, Dipartimento di
Matematica, Piazza di Porta San Donato 5, 40127 Bologna,
Italy.}

\setcounter{section}{0}

\section{Introduction}

A resonator consists of a bounded cavity (the chamber) connected to the exterior by a thin tube (the neck of the chamber). 
The frequencies of the sounds it produces are determined by the shape of the chamber, while their duration by the length and the width of the neck  in a non-obvious way,
and our goal is to understand these. Mathematically, this phenomenon is described by the resonances of the
Dirichlet Laplacian $-\Delta_\Omega$ on the domain $\Omega$ consisting of the union of the chamber, 
the neck and the exterior.

We recall that resonances are the eigenvalues of a complex deformation of $-\Delta_\Omega$;  their real and imaginary
parts are the frequencies and inverses of the half-lives, respectively, of the corresponding vibrational modes. It is of obvious physical 
interest to estimate these two quantities as precisely as  possible. One practical way to do this involves studying 
this problem in the asymptotic limit when the width $\varepsilon$ of the neck tends to zero. Those resonances with imaginary parts 
tending to zero converge to the eigenvalues of the Dirichlet Laplacian on the cavity, and there is an exponentially small upper bound 
for the absolute values of the imaginary parts (the widths) of the resonances.

 Many papers have been devoted to this problem (see, e.g., \cite{Be, Fe, HM, MN1, MN2}, and references therein). For instance, in \cite{Be}, it is proved that, as the width of the neck becomes small, these resonances approach either an eigenvalue of the cavity, or a resonance of the exterior region.  In \cite{Fe}, a construction of states that are concentrated in the cavity is obtained, with lower estimates on their sojourn time (and thus, upper estimates on the widths of resonances, that is, the absolute value of their imaginary parts). In \cite{HM}, a very general exponentially small upper bound on the widths of resonances is proved. 
 
 However, without very 
restrictive hypotheses, no lower bound is known. We mention in particular that lower bounds are known in
the one-dimensional case \cite{Ha, HaSi}. As for the higher dimensional case, we can mention \cite{FL, Bu, HS, FLM} 
which contain results concerning exponentially small widths of quantum resonances, but these do not apply to a Helmholtz resonator (in particular, the semiclassical lower bound obtained in \cite{HS} is optimal).

In \cite{MN1}, an optimal lower bound has been optioned for very particular two-dimensional Helmholtz resonators, for which the exterior consists of an infinite straight half tube. Then, the result has been extended in \cite{MN2} to much more general two-dimensional Helmholtz resonators, under the condition that the neck meets the boundary of the external region perpendicularly  to it, and that the exterior region is concave and symmetric there. Moreover, an extension to larger dimensions (up to 12) was obtained, but only for necks with a square section.

Here, we plan to generalize the result to any $n$-dimensional Helmholtz resonator with straight neck, without particular assumption on the section of the neck or on the boundary near the mouth of the neck.
 (see Theorem \ref{mainth}). 
 
 As in \cite{MN2}, the problem is first related to a lower bound on the resonant function in a large annulus. Assuming, by contradiction, that this function is small there, the smallness can be propagated up to a small neighborhood of the end part of the neck, by means of general Carleman estimates. In order to have a good enough control on the rate of decay near the mouth of the neck, however, we need stronger results. These are furnished to us by the powerful tool of the so-called  ``limiting Carleman weights'' developed in \cite{KSU}. Indeed, they give a framework where more precise Carleman estimates can be done, leading to an exponential decay estimate on the resonant function with a rate going to zero as we approach the neck, but not too quickly. Afterwards, this permits us to propagate the estimate in the inside part of the neck, this time by using explicit  Carleman estimates.
At that point, the contradiction is obtained as in \cite{MN2}, by using a result of \cite{BHM1, BHM2} on the size of the Dirichlet eigenfunctions of the cavity.

\section{Assumptions and results}

Let $\mathcal C$ and $\mathcal B$ be two bounded open sets in $\R^n$ ($n\geq 2$), with $\mathcal C^\infty$ boundary, and denote by $\overline{\mathcal C}$, $\overline{\mathcal B}$
their closures, and by $\partial{\mathcal C}$, 
$\partial{\mathcal B}$ their boundaries. We assume that $\mathcal C$ is connected, $\mathcal B$ is contractible,  and that Euclidean coordinates $x=(x_1,\dots, x_n)=:(x_1,x')$ can be chosen in such a way that, for some 
$L>0$,  one has, 
\be
\label{hyp1}
\overline{\mathcal C}\subset \mathcal B\quad; \quad 0\in \partial{\mathcal C} \quad; \quad (L,0_{\R^{n-1}})\in \partial{\mathcal B}\quad; \quad
{[0,L]}\times\{0_{\R^{n-1}}\}\subset \overline{\mathcal B}\backslash {\mathcal C}.
\ee
We also assume,
\be
\label{hyp2}
{[0,L]}\times\{0_{\R^{n-1}}\} \mbox{ is transversal to } \partial{\mathcal B} \mbox{ at } (L,0_{\R^{n-1}}),\mbox{ and to }\partial{\mathcal C} \mbox{ at } 0_{\R^{n}}.
\ee

Let  $D_1\subset \R^{n-1}$ be a bounded domain containing the origin, with smooth boundary $\partial D_1$. For $\varepsilon >0$ small enough, we set $D_\varepsilon:=\varepsilon D_1$ and,
$$
\begin{aligned}
&{\mathbf E}:= \R^n \backslash \overline{\mathcal B};\\
&{\mathcal T}(\varepsilon):= ([-\varepsilon_0, L+\varepsilon_0]\times D_{\varepsilon}) \cap \left(\R^n\backslash ({\mathbf E}\cup{\mathcal C})\right);\\
&\mathcal C(\varepsilon) = {\mathcal C} \cup \mathcal{T}(\varepsilon),
\end{aligned}
$$
where $\varepsilon_0>0$ is fixed sufficiently small in order that $[-\varepsilon_0, L+\varepsilon_0]\times\{0_{\R^{n-1}}\}$ crosses $\partial{\mathcal C}$ and $\partial{\mathcal B}$ at one point only.
Then,
the resonator is defined as, 
$$\Omega(\varepsilon):={\mathcal C}(\varepsilon) \cup{\mathbf E}.$$

For any domain $Q$, let $P_Q$ denote the Laplacian $-\Delta_Q$ with Dirichlet boundary conditions on $\partial Q$,
and set $P_\varepsilon:= P_{\Omega_\varepsilon}$. 

The resonances of $P_\varepsilon$ are defined as the eigenvalues of the operator obtained by performing a complex dilation with
respect to $x$, for $|x|$ large. 

\begin{remark}\sl
As $\varepsilon \tendsto 0_+$, the resonator $\Omega(\varepsilon)$ collapses to $\Omega_0:={\mathcal C}\cup [0,M_0]
\cup{\mathbf E}$, where $M_0$ is the point $(L,0_{\R^{n-1}})\in\R^n$. In particular, for $\varepsilon =0$, the interior $\mathring{\Omega}_0={\mathcal C}
\cup{\mathbf E}$ of $\Omega_0$ is such that the resonances of $P_{\mathring{\Omega}_0}$ consist of the eigenvalues of $P_{\mathcal C}$ and the resonances of $P_{\mathbf E}$. It is well known that $P_{\mathbf E}$ has no imbedded eigenvalues, and thus its resonances (that are of course independent of $\varepsilon$) stay away from the real line. Moreover, by the results of \cite{Be, HM}, we know that the set of the resonances of $P_{\mathring{\Omega}_0}$ (that includes the eigenvalues of $P_{\mathcal C}$)  is nothing but the limit set of those of $P_\varepsilon$ as $\varepsilon \to 0_+$.
\end{remark}

We are interested in those
resonances of $P_\varepsilon$ that are close to the eigenvalues of $P_{\mathcal C}$. Thus, let $\lambda_0 >0$ be an eigenvalue 
of $P_{\mathcal{C}}$ with $u_0$ the  corresponding normalized eigenfunction. As in \cite{MN2}, we assume,

\be
\label{hyp3}
\lambda_0 \mbox{ is the lowest eigenvalue of } P_{\mathcal C}.
\ee

In the sequels, we denote by $u_0$ the corresponding positive normalized eigenfunction of $P_{\mathcal C}$.

By the arguments of \cite{Be, HM}, we know that there is a unique resonance $\rho(\varepsilon)\in\mathbb{C}$ of  $P_\varepsilon$ such that 
$\rho(\varepsilon)\to \lambda_0$ as $\varepsilon\to 0_+$. Furthermore, denoting by $\alpha_0$ the square root of the first eigenvalue of $-\Delta_{D_1}$, there is an eigenvalue $\lambda(\varepsilon)$ of 
$P_{\mathcal{C}(\varepsilon)}$ such that, for any $\delta>0$, 
\be
\label{vp}
|\rho(\varepsilon)-\lambda(\varepsilon)|\leq C_\delta e^{-2\alpha_0(1-\delta)L/\varepsilon},
\ee
for some $C_\delta>0$ and all sufficiently small $\varepsilon>0$. In particular, since $\lambda (\varepsilon) \in \mathbb R$, 
this gives 
\be
\label{upperbound}
|\im \rho(\varepsilon)|\leq C_\delta e^{-2\alpha_0(1-\delta)L/\varepsilon}.
\ee

We now state our main result.
\begin{theorem}\sl 
\label{mainth}
Under Assumptions (\ref{hyp1})-(\ref{hyp3}), for any $\delta >0$ there exists $C_\delta >0$ such that, for all 
$\varepsilon >0$ small enough, one has,
$$
|\im \rho(\varepsilon)|\geq  \frac1{C_\delta}e^{-2\alpha_0 (1+\delta)L/\varepsilon}.
$$
\end{theorem}  
 \begin{remark}\sl
Gathering (\ref{upperbound}) and Theorem \ref{mainth}, we can reformulate the result as,
\be
\lim_{\varepsilon\to 0_+} \varepsilon \ln |\im \rho(\varepsilon) | = -2\alpha_0 L.
\ee
 \end{remark}
  \begin{remark}\sl
In the case where $\mathcal C$ is not connected (or $\mathcal B$ is not conttractible), the result is still valid under a straightforward additional condition on $\lambda_0$. Indeed, in this case the obstacle may contain several connected cavities (and they may admit $\lambda_0$ as eigenvalue, or not). However, in order that the result remains valid, it is sufficient to assume that the tube ${\mathcal T}(\eps)$ connects the exterior ${\mathbf E}$ with one of the connected cavities admitting $\lambda_0$ as a simple eigenvalue.
 \end{remark}
 
 \section{Background properties}

By definition, the resonance $\rho (\varepsilon)$ is an eigenvalue of
the complex
 distorted operator,
$$
P_\varepsilon (\mu):= U_\mu P_\varepsilon U_\mu^{-1},
$$
where $\mu>0$ is a small parameter, and $U_\mu$ is a complex
distortion of the
 form,
$$
U_\mu \varphi (x):= \varphi (x+i\mu f(x)),
$$
with $f\in C^\infty (\R^n; \R^n)$, $f= 0$ near $\overline{\mathcal B}$, $f(x) =x$ for $|x|
$ large enough. (In principle, $U_\mu$ may act on holomorphic functions $\varphi$ only, but the explicit form of $U_\mu P_\varepsilon U_\mu^{-1}$ as a second order differential operator allows to make act $P_\varepsilon (\mu)$ on $L^2(\R^n)$ with domain $H^2(\R^n)$ ; 
moreover, by Weyl Perturbation Theorem, the essential spectrum of
$P_\varepsilon (\mu)$ 
is  $e^{-2i\theta}\R_+$,
 with $\theta =\arctan\mu$.) 

It is well known that such eigenvalues do not depend on $\mu$ (see,
e.g., \cite{SZ, HeM}), and that the corresponding eigenfunctions are
of the form $U_\mu u_\varepsilon$
 with $u_\varepsilon$ independent of $\mu$, smooth on $\R^n$ and
 analytic in a complex sector around ${\mathbf E}$. In other words, $u_\varepsilon$ is a non trivial analytic solution of the equation $-\Delta u_\varepsilon = \rho (\varepsilon) u_\varepsilon$ in $\Omega(\varepsilon)$, such that $u_\varepsilon\left\vert_{\partial\Omega(\varepsilon)}\right. =0$ and, for all $\mu >0$ small enough,  $U_\mu u_\varepsilon$ is well defined and is in $L^2(\Omega(\varepsilon))$ (in our context, this latter property will be taken as a definition of the fact that $u_\varepsilon$ is {\it outgoing}). Moreover,
 $u_\varepsilon$ can be 
normalized by setting, for some fixed $\mu>0$,
$$
\Vert U_\mu u_\varepsilon\Vert_{L^2(\Omega(\varepsilon ))} =1.
$$

In that case, we learn from \cite{HM} (in particular Proposition 3.1
and formula
 (5.13)), that, for any $\delta >0$, and for any $R>0$ large enough, one has,
\be\label{ugrand}
\|u_\varepsilon\|_{L^2(\Omega(\varepsilon)\cap\{ |x|<R\})}= 1+ \O(e^{(\delta-\alpha_0)L/\varepsilon}),
\ee and 
\be
\label{upetit}
\|u_\varepsilon\|_{H^1( {\mathbf E}\cap\{ |x|<R\})}=  \O(e^{(\delta-\alpha_0))L/\varepsilon}).
\ee

Now, we take $R>0$ such that $\overline{\mathcal B} \subset \{ |x| <R\}$. Using the equation $-\Delta u_\varepsilon = \rho u_\varepsilon$
and Green's
 formula on the domain $\Omega(\varepsilon)\cap\{ |x|<R\}$, and using polar coordinates $(r,\omega)$, we obtain,
$$
\im\rho \int_{\Omega(\varepsilon)\cap\{ |x|<R\}  } |u_\varepsilon|^2dx = -\im \int_{\S^{n-1}} \frac {\partial u_\varepsilon}{\partial r}(R,\omega)\overline u_\varepsilon (R,\omega)R^{n-1}d\sigma_{n-1}(\omega),
$$
(where $d\sigma_{n-1}(\omega)$ stands for the surface measure on $\S^{n-1}$),
and thus, by (\ref{ugrand}), and for any $\delta >0$,
\be
 \label{green}
\im\rho =-(1+\O(e^{(\delta-2\alpha_0)L/\varepsilon}))\,\,\im \int_{\S^{n-1}} \frac {\partial u_\varepsilon}{\partial r}(R,\omega)\overline u_\varepsilon (R,\omega)R^{n-1}d\sigma_{n-1}(\omega)\ee
where the $\O$ is locally uniform with respect to $R$.

Therefore, to prove our result, it is sufficient to obtain a lower
bound on 
$\im \int_{\S^{n-1}} \frac {\partial u_\varepsilon}{\partial r}(R,\omega)\overline u_\varepsilon (R,\omega)R^{n-1}d\sigma_{n-1}(\omega)$. 
Note that, by using (\ref{upetit}), we immediately obtain (\ref{upperbound}).

Starting from this formula, the following proposition has been proved in \cite{MN2} (the proof is actually done in 2 dimensions only, but can be generalized easily to any dimension: see \cite{MN2}, Remark 4.6):
\begin{proposition}[Martinez-N\'ed\'elec \cite{MN2}]\sl
\label{LDisc} \sl Let $R_1>R_0>0$ be fixed in such a way that  $\overline{\mathcal B}\subset \{ |x| < R_0\}$. Then, for any $C>0$, there exists a constant $C'=C'(R_0, R_1,C)>0$ such that,  for all $\varepsilon >0$ small enough, one has,
$$
|\im\rho| \geq \frac1{C'}\Vert u_\varepsilon\Vert^2_{L^2(R_0< |x| < R_1)} - C'e^{-C/\varepsilon}.
$$
\end{proposition}

Then, reasoning by contradiction as in \cite{MN2}, we assume the existence of $\delta_0 >0$ such that, along a sequence $\varepsilon \to 0_+$, one has
\be
\label{absurd}
|\im\rho| =\O(e^{-2(\alpha_0+\delta_0)L/\varepsilon}).
\ee 
Proposition \ref{LDisc} (added to standard Sobolev estimates) tells us that for any $R_1>R_0>0$ such that $\overline{\mathcal B}\subset \{ |x| < R_0\}$, we have,
\be
\label{decu}
\Vert u_\varepsilon\Vert_{H^1(R_0<|x| <R_1)} = \O(e^{-(\alpha_0+\delta_0)L/\varepsilon}).
\ee
Still following the procedure used in \cite{MN2}, we see that this estimate can be propagated up to the boundary of $\mathcal B$, away from an arbitrarily small neighborhood of $M_0:=(L,0_{\R^{n-1}})$ (this is done by means of Carleman inequalities up to the boundary \cite{LR, LL}), and one obtains (see \cite{MN2}, Proposition 6.1),
\begin{proposition}[Martinez-N\'ed\'elec \cite{MN2}]\sl
\label{estbound}
Under the assumption (\ref{absurd}), for any neighborhood $\mathcal U$ of $M_0$ and any compact set $K\subset \R^n$, there exists $\delta_K >0$ such that,
$$
\Vert u_\varepsilon\Vert_{H^1({\mathbf E}\cap K \backslash {\mathcal U})} =\O(e^{-(\alpha_0+\delta_K)L/\varepsilon}),
$$
uniformly as $\varepsilon \to 0_+$.
\end{proposition}
From this point, the proof starts to differ completely from that of \cite{MN2}. As a first step, we will improve this estimate by obtaining a control on the rate of decay when we come close to $M_0$. This will be achieved by using again Carleman estimates, but in the spirit of \cite{KSU}. 
 The final step will consist in propagating the estimate inside the thin tube, by using some ``hand-made'' Carleman estimate on such an $\varepsilon$-dependent domain. After that, the proof can be completed exactly as in \cite{MN2}.

\section{Estimate near $M_0$}

In order to extend the previous estimate close to $M_0$ with a good enough control on the rate of decay, we use the ideas of \cite{KSU}, in particular the notion of limiting Carleman weight. We will prove:
\begin{proposition}
\label{P_superlinear}
 Let $R>0$ and $s\in (0,1)$. Under assumption (\ref{absurd}), there exists $\delta>0$ such that for all $\nu\in (0,R)$, one has,
 $$ \|u_\eps\|_{H^1\left(\left\{\nu<|x-M_0|<R\right\}\cap \textbf{E}\right)}=\O(e^{-(\alpha_0L+\delta\nu^s)/\eps}),$$
 uniformly for $\eps >0$ small enough.
\end{proposition}
\begin{remark}\sl
In Section 5, we will use this proposition with $s=\frac12$.
\end{remark}
\begin{proof}
We can assume $R>1$. Let $q\in (0,1)$ be a small parameter depending on $s$ and the geometry, to be specified later. It is sufficient to prove that there exists $\delta>0$ such that for all $k\geq 1$, 
\be
 \label{induction}
\|u_\eps\|_{H^1\left(\left\{q^k<|x-M_0|<R\right\}\cap \textbf{E}\right)}=\O(e^{-(\alpha_0L+\delta q^{ks})/\eps})
\ee
 uniformly for $\eps >0$ small enough.
We prove \eqref{induction} by induction. Proposition \ref{estbound} proves that there exists $\delta>0$ such that \eqref{induction} holds with $k=1$.

We next assume \eqref{induction} for some $k\geq 1$. We  set,
$$M_1:=(L-\frac18q^{k+1}, 0_{\R^{n-1}})=M_0-(\frac18q^{k+1}, 0_{\R^{n-1}}).$$
For $x\in \R^n$, we define,
$$
\varphi(x):= \ln |x-M_1|,
$$
and we denote by
$$
a(x,\xi ):= \xi^2-(\nabla\varphi (x))^2\quad ;\quad b(x,\xi):= 2\nabla\varphi (x)\cdot \xi
$$
the real and imaginary parts of the principal symbol of the semiclassical operator $e^{-\varphi /h}(-h^2\Delta)e^{\varphi /h}$. Then, as observed in \cite{DKSU},  $\varphi$ is a limiting Carleman weight on ${\mathbf E}$ in the sense of \cite{KSU}, that is $|\nabla\varphi|\not =0$ on ${\mathbf E}$,  and $\{a,b\}=0$ on the set $\{ a=b=0\}$ (where $\{a,b\}$ stands for the Poisson bracket between $a$ and $b$).

Moreover, thanks to (\ref{hyp2}), for $x\in\partial{\mathcal B}$ such that $\left|x-M_0\right|\leq \frac{1}{C}q^{k/2}$, where $C>0$ is a constant given by the geometry of $\mathcal{B}$, we have
$$
\nabla\varphi (x)\cdot \vec{n} <0,
$$
where $\vec{n}$ stands for the outward pointing unit normal to ${\mathbf E}$ (or, equivalently, the inward pointing unit normal to ${\mathcal B}$).
We can take $q$ small, so that $3q^k\leq \frac{1}{C}q^{k/2}$ for any $k\geq 1$.
Applying Proposition 3.2 of \cite{KSU} with an open set of the form,
$$
V:= {\mathbf E}\cap \{ |x-M_1| < 3q^k\},
$$
we see that, for any $v\in C^\infty (\overline{V})\cap C_0^\infty (\overline{\mathbf E}\cap \{ |x-M_1| < R\})$ such that $v\left|_{\partial{\mathbf E}}\right. =0$, 
one has,
\be
\|e^{\varphi /h}v\|^2_{L^2(V)} + h^2\|e^{\varphi /h}\nabla v\|^2_{L^2(V)}\leq C_0h^2\|e^{\varphi /h}\Delta v\|^2_{L^2(V)},
\ee
where the positive constant $C_0$ does not depend on $v$ and on $h>0$ small enough (but might depend on $k$).

Next, we apply this estimate with $v(x)=\chi (|x-M_1|) u_\eps(x)\indic_{\mathbf{E}}$, where 
$$\chi\in C_0^\infty \left(\left(\frac{q^{k+1}}{4},3q^k\right);[0,1]\right),  \quad \chi =1 \text{ on }\left[\frac{q^{k+1}}{2},2q^k\right],$$
and $\indic_{\mathbf{E}}$ stands for the characteristic function of $\mathbf{E}$.
We also take $h:=\varepsilon/\mu_k$, where $\mu_k>0$ will be fixed later on. Setting $r:=|x-M_1|$, we obtain
$$
\| r^{\mu_k/\varepsilon}v\|^2+\varepsilon^2\mu_k^{-2}\|r^{\mu_k/\varepsilon} \nabla v\|^2\leq C_0\varepsilon^2\mu_k^{-2}\|r^{\mu_k/\varepsilon}([\Delta, \chi]u-\rho v)\|^2,
$$
and thus, for $\varepsilon$ small enough,
\be
\label{Car1}
\| r^{\mu_k/\varepsilon}v\|^2+\varepsilon^2\mu_k^{-2}\|r^{\mu_k/\varepsilon} \nabla v\|^2\leq 2C_0\varepsilon^2\mu_k^{-2}\|r^{\mu_k/\varepsilon}[\Delta, \chi]u\|^2.
\ee
Now, we have 
$$ \mathrm{Supp}\,\nabla\chi \subset \left\{r\in \left[\frac{q^{k+1}}{4},\frac{q^{k+1}}{2}\right]\cup \left[2q^k,3q^k\right]\right\}.$$ 
Note that by the definition of $M_1$, $|x-M_1|\geq 2q^k\Longrightarrow |x-M_0|\geq q^k$. By the induction hypothesis (here and in the sequel $C$ denotes a large positive constant depending on $k$, that may change from line to line)
$$ \|r^{\mu_k/\varepsilon}[\Delta, \chi]u\|^2_{L^2\left(\left\{2q^k\leq r\leq 3q^k\right\}\cap \textbf{E}\right)} \leq C (3q^k)^{2\mu_k/\eps}e^{-(2\alpha_0L+2\delta q^{ks})/\eps}.$$
By the a priori estimate (\ref{upetit}) on $u$, for any $\delta'>0$ there exists $C$ such that for small $\eps>0$,
$$ \|r^{\mu_k/\varepsilon}[\Delta, \chi]u\|^2_{L^2\left(\left\{q^{k+1}/4\leq r\leq q^{k+1}/2\right\}\cap \textbf{E}\right)} \leq C \left( \frac{q^{k+1}}{2} \right)^{2\mu_k/\eps}e^{-(2\alpha_0L-2\delta')/\eps}. $$
We take $\delta'=\mu_k\log \frac 65$ and obtain
$$ \|r^{\mu_k/\varepsilon}[\Delta, \chi]u\|^2_{L^2\left(\left\{q^{k+1}/4\leq r\leq q^{k+1}/2\right\}\cap \textbf{E}\right)} \leq C \left( \frac 35 q^{k+1} \right)^{2\mu_k/\eps}e^{-2\alpha_0L/\eps}.$$
Combining with (\ref{Car1}) we obtain that for small $\eps>0$,
\begin{multline*}
\left( \frac{3}{4}q^{k+1} \right)^{2\mu_k/\eps} \|u\|^2_{H^1\left(\left\{\frac 34 q^{k+1} \leq r\leq 2q^k\right\}\cap \textbf{E}\right)}\\
\leq C\left[\left( \frac 35q^{k+1} \right)^{2\mu_k/\eps}+\left(3q^k\right)^{2\mu_k/\eps}e^{-\frac{2\delta}{\eps} q^{ks}} \right] e^{-2\alpha_0L/\eps},
\end{multline*}
that is 
\begin{multline}
\label{Car2}
\|u\|^2_{H^1\left(\left\{\frac 34 q^{k+1} \leq r\leq 2q^k\right\}\cap \textbf{E}\right)}\\
\leq C\left[\left( \frac 45\right)^{2\mu_k/\eps}+\left(\frac 4q\right)^{2\mu_k/\eps}e^{-\frac{2\delta}{\eps} q^{ks}} \right] e^{-2\alpha_0L/\eps}.
\end{multline}
At this point, we fix $\mu_k$ in such a way that one has,
$$\left( \frac 45\right)^{2\mu_k/\eps}=\left(\frac 4q\right)^{2\mu_k/\eps}e^{-\frac{2\delta}{\eps} q^{ks}}$$
This gives $\mu_k := \frac{\delta q^{ks}}{\log 5-\log q} $
(which is indeed $>0$ if $q$ is small), and (\ref{Car2}) becomes,
\begin{equation}
\label{Car3}
\|u\|^2_{H^1\left(\left\{\frac 34 q^{k+1} \leq r\leq 2q^k\right\}\cap \textbf{E}\right)}
\leq Ce^{\frac{2\delta q^{ks}\log\frac{4}{5}}{\eps\log \frac{5}{q}}}e^{-2\alpha_0L/\eps}.
\end{equation} 
We next choose $q$ so small that $\log \frac 54\geq q^s\log \frac 5q$. This yields 
$$\frac{2\delta q^{ks}\log\frac{4}{5}}{\eps\log \frac{5}{q}}     \leq -\frac{2q^{(k+1)s}\delta}{\eps}.$$
From \eqref{Car3}, we deduce
\begin{equation*}
\|u\|^2_{H^1\left(\left\{\frac 34 q^{k+1} \leq r\leq 2q^k\right\}\cap \textbf{E}\right)}
\leq C e^{-2\delta q^{(k+1)s}/\eps}e^{-2\alpha_0L/\eps}.
\end{equation*} 
Since $q^{k+1}\leq |x-M_0|\leq q^k\Longrightarrow \frac{3}{4}q^{k+1}\leq |x-M_1|\leq 2q^k$, we obtain
\begin{equation}
\label{Car4}
\|u\|^2_{H^1\left(\left\{q^{k+1} \leq |x-M_0|\leq q^k\right\}\cap \textbf{E}\right)}
\leq C e^{-2\delta q^{(k+1)s}/\eps}e^{-2\alpha_0L/\eps}.
\end{equation} 
By the induction hypothesis, we can replace $q^{k+1}\leq |x-M_0|\leq q^k$ in the left-hand side of (\ref{Car4}) by $q^{k+1}\leq |x-M_0|\leq R$, concluding the proof of (\ref{induction}) at rank $k+1$.
\end{proof}

% By slightly modifying the notations, in particular we have proved,
% \begin{proposition}\sl For any $r_0>r_1>\varepsilon_1$ all small enough, there exists $\delta_1 >\alpha_0(r_1-\varepsilon_1)$ such that,
% $$
% \|u\|_{H^1( r_1< |x-M_1|< r_0\}}= \O(e^{-(\alpha_0L+\delta_1)/\varepsilon}),
% $$
% uniformly as $\varepsilon \to 0_+$. Here, $M_1:=(L-\varepsilon_1,0_{\R^{n-1}})$.
% \end{proposition}

\section{Estimate inside the neck}
In this section we prove:
\begin{proposition}
\label{P_neck}
Let $r_0\in (0,L)$, and assume (\ref{absurd}). Then, there exists a small constant $\delta_1>0$ such that, 
 \be
 \label{troppetit}
 \|u_{\eps}\|_{H^1([r_0,L]\times D_{\eps})}=\O(e^{-(\alpha_0+\delta_1)r_0/\eps}),
 \ee
 uniformly for $\eps >0$ small enough.
\end{proposition}
The proof relies on the following Carleman inequality:
\begin{lemma}
\label{L_Carleman} Let $\eta_0>0$ be small enough.
 There exists $C>0$ such that for all $\alpha_1>0$ and $\eps>0$, for all $v\in C^{\infty}(\overline{\Omega(\eps)})$ verifying,
 \begin{gather}
  \label{neck1}
  v_{\restriction \partial\Omega(\eps)}=0;\\
  \label{neck2}
 {\rm Supp}\, v\,\subset \,\{x_1>0\};\\
  \label{neck3}
  {\rm Supp}\, v\cap {\mathbf E}\,\subset\, \{|x-M_0|< \eta_0\};
 \end{gather}
one has,
\begin{equation}
 \label{neck4}
 \frac{\alpha_1}{\eps}\|e^{\alpha_1x_1/\eps}v\|_{L^2(\Omega(\eps))}\leq C\|e^{\alpha_1x_1/\eps}\Delta v\|_{L^2(\Omega(\eps))}.
\end{equation} 
\end{lemma}
\begin{proof}[Proof of the lemma]
We start with a general computation, valid for any open subset $\Omega$ of $\mathbb{R}^n$ with smooth boundary, and any function $v\in C^{\infty}(\overline{\Omega})$ satisfying Dirichlet boundary conditions. Let $\varphi\in C^{4}(\overline{\Omega})$ be real-valued. We will compute $\|e^{\varphi}\Delta v\|$. 

We have $e^{\varphi}\Delta v=e^{\varphi}\Delta e^{-\varphi}w$, where $w=e^{\varphi}v$. Moreover, $-e^{\varphi}\Delta e^{-\varphi}=A+iB$, where $A$ and $B$ are the formally self-adjoint operators given by 
$$Aw:=-\Delta w -|\nabla \varphi|^2w,\qquad  iBw:=2(\nabla\varphi)\cdot\nabla w+(\Delta \varphi )w.$$
Then,
\begin{multline}
 \label{neck6}
 \|e^{\varphi} \Delta v \|^2\\
=\|(A+ iB)w\|^2=\|Aw\|^2+\|Bw\|^2+2\textrm{Re}\int_{\Omega}Aw\, \overline{iBw}\,dx.
\end{multline} 
We claim:
\begin{multline}
\label{neck7}
 2\re\int_{\Omega}Aw\, \overline{iBw}\,dx=2\int_{\Omega}\nabla \left( |\nabla \varphi|^2 \right)\cdot \nabla \varphi\, |w|^2\\ +4 \int_{\Omega} (\nabla w)^T \varphi''(x)\nabla \overline{w}
 -\int_{\Omega}\Delta^2\varphi\, |w|^2-2\int_{\partial\Omega}\partial_n\varphi|\partial_n w|^2,
\end{multline}
where $\partial_n$ is the outward pointing normal unit at the boundary of $\Omega$, $\varphi''$ is the Hessian matrix of $\varphi$, and $(\nabla w)^T$ is the transpose of $\nabla w$. 

Indeed, expanding $A w$ and $iB w$, we obtain:
\begin{align}
\label{commutator}
2\re\int_{\Omega}Aw\, \overline{iBw}
&=
-4\re \int_{\Omega} \Delta w\nabla \varphi\cdot\nabla \wbar-2\re \int_{\Omega}\Delta w\Delta\varphi \wbar
\\ 
\notag
&\quad \,-4\re\int_{\Omega} |\nabla \varphi|^2w\nabla \varphi\cdot\nabla \wbar-2\int_{\Omega}|\nabla \varphi|^2\Delta\varphi\, |w|^2\\
\notag
&=I_1+I_2+I_3+I_4.
\end{align}
We compute the terms $I_j$, $j=1,2,3$. Using the Green formula and the Dirichlet boundary condition for $w$, we obtain,
\begin{align*}
 I_1&=4 \re \int_{\Omega} \nabla w\cdot \nabla (\nabla \varphi \cdot\nabla \wbar)-4\re \int_{\partial \Omega} \partial_n w\nabla \varphi\cdot\nabla \wbar.
\end{align*}
Since $w_{\restriction \partial\Omega}=0$, one has $\nabla\varphi\cdot\nabla \wbar=\partial_n \varphi\cdot\partial_n \wbar$ on $\partial \Omega$.
Moreover,
\begin{align*}
\re \nabla w\cdot \nabla (\nabla \varphi\cdot\nabla \wbar) &= \sum_{1\leq j,k\leq n} \frac{\partial w}{\partial x_j}\frac{\partial^2\varphi}{\partial x_j\partial x_k}\frac{\partial \wbar}{\partial x_k}+\frac{1}{2}\sum_{1\leq j,k\leq n} \frac{\partial \varphi}{\partial x_k}\frac{\partial}{\partial x_k} \left|\frac{\partial w}{\partial x_j}\right|^2
\\
&=  (\nabla w)^T \varphi''\nabla \wbar +\frac{1}{2}\nabla \varphi\cdot\nabla(|\nabla w|^2).
\end{align*}
Thus,
\begin{equation}
\label{I_1} 
I_1=4 \re \int_{\Omega} (\nabla w)^T \varphi'' \nabla \wbar-2\int_{\Omega} \Delta \varphi |\nabla w|^2-2\int_{\partial\Omega} \partial_n\varphi|\partial w|^2.
\end{equation}
On the other hand,
\begin{equation*}
 I_2=2\int_{\Omega} |\nabla w|^2\Delta \varphi +2\re \int_{\Omega} \nabla(\Delta \varphi)\cdot \nabla w\,\wbar.
\end{equation*}
Since $\re \nabla w\,\wbar=\frac{1}{2}\nabla |w|^2$, this yields
\begin{equation}
 \label{I_2}
I_2=2\int_{\Omega} |\nabla w|^2\Delta \varphi -\int_{\Omega} \Delta^2\varphi |w|^2.
\end{equation}
Writing $I_3=-2\int_{\Omega} |\nabla \varphi|^2\nabla \varphi\cdot\nabla |w|^2$, we obtain
\begin{equation}
 \label{I_3}
I_3=2\int_{\Omega} |\nabla \varphi|^2\Delta \varphi \,|w|^2+2\int_{\Omega}\nabla\left(|\nabla \varphi|^2\right)\cdot\nabla \varphi\,|w|^2.
\end{equation}
Combining (\ref{commutator}), (\ref{I_1}), (\ref{I_2}) and (\ref{I_3}), we deduce (\ref{neck7}). When $\varphi$ depends only on the variable $x_1$, (\ref{neck7}) becomes:
\begin{multline}
 \label{neck7'}
 2\re\int_{\Omega}Aw\, \overline{iBw}=4\int_{\Omega} \varphi''|\partial_{x_1} w|^2+4\int_{\Omega}\varphi''(\varphi')^2|w|^2\\ 
 -\int_{\Omega}\varphi^{(4)}\,|w|^2-2\int_{\partial\Omega}n_1\varphi'|\partial_n w|^2,
\end{multline} 
where $n_1$ is the first coordinate of the outward pointing unit vector $\vec{n}$ at the boundary of $\Omega$.

We use (\ref{neck6}), (\ref{neck7'}) with $\Omega=\Omega(\eps)$, $v$ satisfying the assumptions of the lemma, and 
\begin{equation}
 \label{neck5}
 \varphi(x_1)=\frac{\alpha_1}{\eps}x_1+\frac{x_1^2}{2}.
\end{equation} 
Note that $x_1$ is a limiting Carleman weight. The addition of the strictly convex term $\frac{x_1^2}{2}$ is in the spirit of \cite{KSU}. However we cannot directly use the results of \cite{KSU} since the domain of integration $\Omega(\eps)$ depends on $\eps$.

Note that $n_1=0$ at the boundary of the neck $[0,L]\times D_{\eps}$. Chosing $\eta_0$ small, we have $n_1\varphi'(x_1)=n_1(\alpha_1/\eps+x_1)\leq 0$ on the intersection of the support of $w$ and $\partial\Omega(\eps)$. Moreover $\varphi^{(4)}=0$, $\varphi''$ is nonnegative and 
$$\varphi''(x_1)(\varphi'(x_1))^2=(\alpha_1/\eps+x_1)^2\geq \alpha_1^2/\eps^2.$$
Thus, we see that (\ref{neck7'}) implies:
$$ 2\textrm{Re}\int_{\Omega(\eps)}Aw\, \overline{iBw}\,dx\geq 4\frac{\alpha_1^2}{\eps^2}\int_{\Omega(\eps)} |v|^2e^{2\alpha_1x_1/\eps+x_1^2}.$$
Hence, in view of (\ref{neck6}),
$$\| e^{\alpha_1x_1/\eps+x_1^2/2}\Delta v\|\geq \frac{2\alpha_1}{\eps} \|e^{\alpha_1x_1/\eps+x_1^2/2}v\|. $$
Using that $x_1^2\leq (L+\eta_0)^2$ on the support of $v$, we obtain the conclusion of the lemma.
\end{proof}
\begin{proof}[Proof of Proposition \ref{P_neck}]
 We fix a small $\eps_0>0$ and let $\chi\in C_0^{\infty}(\overline{\Omega(\eps_0)})$ such that 
\begin{equation*}
 \begin{cases}
  \chi(x)=0 & \text{ if }x\in \mathcal{C}\text{ or }x\in [0,\nu]\times D_{\eps_0}\text{ or } \left(x\in \textbf{E} \text{ and } |x-M_0|\geq \nu+\nu^2\right)\\
  \chi(x)=1 & \text{ if }x\in [2\nu,L]\times D_{\eps_0}\text{ or } \left(x\in \textbf{E} \text{ and } |x-M_0|\leq \nu\right), 
 \end{cases}
\end{equation*}
where $\nu>0$ is a small parameter to be specified later. Of course, we can also consider $\chi$ as an element of $C_0^{\infty}(\overline{\Omega(\eps)})$ for $0<\eps<\eps_0$. Using Lemma \ref{L_Carleman} with $\alpha_1=\alpha_0+\nu^{3/4}$ and $v=\chi u$, we obtain, for small $\eps>0$, denoting by $C_{\nu}$ a constant depending on $\nu$ and changing from line to line:
\begin{multline}
 \label{neck8}
\frac{1}{\eps^2}  \|\chi u e^{(\alpha_0+\nu^{3/4})x_1/\eps}\|^2\\
\leq C_{\nu} \|[\Delta,\chi]ue^{(\alpha_0+\nu^{3/4})x_1/\eps}\|^2+C_{\nu}|\rho(\eps)|^2\|\chi u e^{(\alpha_0+\nu^{3/4})x_1/\eps}\|^2.
\end{multline} 
Since $\rho(\eps)$ is bounded uniformly as $\eps\to 0$, the last term of (\ref{neck8}) can be absorbed by its left-hand side for small $\eps$ Moreover
\begin{equation}
 \label{neck9}
\|[\Delta,\chi]ue^{(\alpha_0+\nu^{3/4})x_1/\eps}\|^2\leq C_{\nu} \int_{\omega_1\cup \omega_2}\left(|\nabla u|^2+|u|^2\right)e^{2(\alpha_0+\nu^{3/4})x_1/\eps}\,dx
\end{equation} 
where 
\begin{align*}
\omega_1&:=\{x=(x_1,x')\; :\; \nu\leq x_1\leq 2\nu,\; x'\in D_{\eps} \}\\
\omega_2&:=\{x\in \textbf{E}\;:\; \nu\leq |x-M_0|\leq \nu+\nu^2\}. 
\end{align*}
Since $u$ remains locally bounded in $H^1(\Omega(\eps))$ as $\eps \to 0$, we can bound from above the integral on $\omega_1$ in (\ref{neck9}) by $C_{\nu}e^{4(\alpha_0+\nu^{3/4})\nu/\eps}$. Using Proposition \ref{P_superlinear} with $s=\frac12$, and chosing $\nu$ small enough, we bound from above the integral on $\omega_2$ by 
$$C_{\nu}e^{2(\alpha_0+\nu^{3/4})(L+\nu+\nu^2)/\eps}e^{-(2\alpha_0 L+2\delta \nu^{1/2})/\eps}\leq C_{\nu}e^{-\delta \nu^{1/2}/\eps},$$
as $\eps\to 0$. From (\ref{neck8}), (\ref{neck9}) and the estimates above,
$$\frac{1}{\eps^2} \|u\|_{L^2([r_0,L]\times D_{\eps})}^2e^{2(\alpha_0+\nu^{3/4})r_0/\eps}\leq C_{\nu} e^{4(\alpha_0+\nu^{3/4})\nu/\eps}.$$
Taking $\nu$ small enough, we have $\nu^{3/4}r_0\geq 4(\alpha_0+\nu^{3/4})\nu$, and we deduce 
$$\|u\|_{L^2([r_0,L]\times D_{\eps})}^2 \leq C_{\nu}e^{-(2\alpha_0+\nu^{3/4})r_0/\eps}.$$
Using the equation $-\Delta \chi u =\rho \chi u +[\chi, \Delta]u$ with $\chi$ as before, and performing the scalar product with $\chi u$, we also deduce the same type of estimate for $\nabla u$, and this yields the conclusion of the proposition.
\end{proof}

\section{Completion of the proof of Theorem \ref{mainth}}

At this point, the completion of the proof is entirely taken from \cite{BHM1} (see also \cite{BHM2, MN1, MN2}). For the sake of completeness, here we recall the main arguments (actually, in our setting they are a little bit simpler since we do not have to care too much about negative powers of $\eps$). At first, using Assumption (\ref{hyp3}), we see that the eigenvalue $\lambda (\varepsilon)$ appearing in (\ref{vp}) is necessarily the first Dirichlet eigenvalue of ${\mathcal C}(\eps)$, and, denoting by $v_\eps$ the corresponding normalized positive eigenfunction, by \cite{HM} we know that, for all $s\geq 0$ and $\delta >0$, one has,
$$
\|u_\eps -v_\eps\|_{H^s({\mathcal C}(\eps))} =\O(e^{-(\alpha_0L-\delta)/\eps}),
$$
uniformly for $\eps >0$ small enough.

Then we use the explicit representation of $v=v_\eps$ inside the tube in terms of the transversal modes, 
\be
\label{reptube}
v=\sum_{k=0}^{+\infty} v_k(x_1)\varphi_k(x'/\varepsilon)=\sum_{k=0}^{+\infty} (a_k^+ e^{\theta_k (x_1-\mu)/\varepsilon }+ a_k^- e^{-\theta_k (x_1-\mu)/\varepsilon })\varphi_k(x'/\varepsilon),
\ee
where $a_k^{\pm}\in \R$, $\theta_k:=\sqrt{\alpha_k^2 -\varepsilon^2\lambda(\varepsilon)}$, $\alpha_k^2 :=(k+1)$-th eigenvalue of $-\Delta_{D_1}$ ($\alpha_k>0$), $\varphi_k:=$ normalized $(k+1)$-th eigenfunction of $-\Delta_{D_1}$, $x':=(x_2,\dots, x_n)$, and $\mu:=C_0\eps$ with $C_0>0$ sufficiently large such that $\{\mu\}\times D_\eps\, \cap {\mathcal C}=\emptyset$.

Using standard Sobolev estimates, we  see that $|a_k^{\pm}| =\O(\eps^{-N_0})$ for some $N_0\geq 0$ independent of $k$, and by the Dirichlet boundary condition at $x_1=L+\varepsilon_0$, we deduce that $|a_k^+|=\O(\eps^{-N_0}e^{-2\alpha_kL/\eps})$ uniformly with respect to $k$. As a consequence, for any $x_1\in [\mu,L]$, we obtain,
\be
\label{vdex1}
v_k(x_1)=v_k(\mu)e^{-\theta_k (x_1-\mu)/\eps}+\O(\eps^{-N_0}e^{-\alpha_kL/\eps}).
\ee

Now, using Proposition \ref{P_neck}, we see that, for any $r_0\in (0,L)$, the coefficients $a_k^{\pm}$'s satisfy,
$$
\sum_{k\geq 0}( |a_k^+|^2 e^{2\alpha_k L/\varepsilon} +  |a_k^-|^2 e^{-2\alpha_k r_0/\varepsilon}) =\O(e^{-2(\alpha_0+\delta_1)r_0/\eps}).
$$

In particular, $a_0^\pm=\O(e^{-\delta_1/\eps})$, and thus, 
$$
v_0(\mu )=\O(e^{-\delta_1/\eps}).
$$
Moreover, by Weyl's law and a general result of \cite{Da}, we know that,
$$
\alpha_k \sim k^{\frac2{n}}\quad ;\quad \sup_{D_1}\frac{|\varphi_k|}{\varphi_0} =\O((\alpha_k)^{\frac{n}2}) \quad (k\to\infty).
$$
 Writing,
$$
v_k(x_1)=\eps^{n-1}\int_{D_1}v(x_1, \eps y)\varphi_k(y)dy=\eps^{n-1}\int_{D_1}v(x_1, \eps y)\varphi_0(y)\frac{\varphi_k(y)}{\varphi_0(y)}dy
$$
and taking advantage of the fact that $v$ and $\varphi_0$ are non negative, we deduce,
$$
v_k(\mu)=\O((k+1)e^{-\delta_1/\eps}).
$$
Inserting in (\ref{vdex1}), and taking $C>C_0$, we obtain,
$$
\| v\|_{L^2([C\eps,L]\times D_\eps)} =\O(e^{-\delta_1/\eps}).
$$

As a consequence, using the equation $-\Delta v =\lambda v$ and standard Sobolev inequalities, we obtain the existence of a constant $\delta_2 >0$ such that, for $C>0$ large enough,
$$
\sup_{[C\varepsilon,L]\times D_{\eps}}|v_\varepsilon| =\O(e^{-\delta_2/\eps}).
$$

Then, using the boundary Harnack inequality for non-negative solutions to  parabolic equations (see, e.g., \cite{FGS, BHM2}), we deduce,
$$
\sup_{{\mathcal C}(\eps)\cap\{|x|\leq C\eps\}}|v_\eps|=\O(e^{-\delta_2/\eps}),
$$ 
and thus, for any $C>0$ large enough,
$$
\|v_\eps\|_{L^2({\mathcal C}(\eps)\cap\{|x|\leq C\eps\})}=\O(e^{-\delta_2/\eps}).
$$
Taking $\chi_\eps \in  C_0^\infty (\{|x| <2C\eps\})$, $\chi_\eps(x)=1$ for $|x|\leq C\eps$, $\partial^\alpha \chi_\eps=\O(\eps^{-|\alpha|})$, and using again the equation $-\Delta \chi_\eps v=\lambda \chi_\eps v+[\chi_\eps,\Delta]v$, we see that a similar estimates holds for $|\nabla v|$. Then, fixing $C$ sufficiently large, and defining,
$$
w:= (1-\chi_\eps(x)) v,
$$
we have $w\in (H^2\cap H^1_0)({\mathcal C})$, and,
$$
-\Delta_{\mathcal C'} w = \lambda(\eps) w + r\quad ; \quad \| r\|_{L^2({\mathcal C})} = \O(e^{-\delta_2/\eps}).
$$
Using the fact that $\lambda (\varepsilon)\to \lambda_0$ as $\eps\to 0$, and that $\lambda_0$ is simple, we easily deduce (by writing the spectral projector on $u_0$ as a contour integral of the resolvent),
$$
w = u_0 + \O(e^{-\delta_2/\eps})
$$
in $L^2({\mathcal C})$. In particular,
\be
\label{estexp0}
\|u_0\|_{L^2({\mathcal C}\cap\{|x|\leq C\eps\})} =\O(e^{-\delta_2/\eps}).
\ee
However, since $u_0 >0$ in ${\mathcal C}$, a well known result of Hopf (see, e.g., \cite{GiTr}) insures that $\partial_n u (0) <0$ (where $\partial_n$ stands for the outer normal derivative of $\mathcal C$). As a consequence, one necessarily has $u_0(x)\geq cd(x,\partial{\mathcal C})$ near 0 (with some $c>0$), and thus,
\be
\|u_0\|_{L^2({\mathcal C}\cap\{|x|\leq C\eps\})} \geq c \eps^{1+\frac{n}2}.
\ee
Estimate (\ref{estexp0}) is clearly in contradiction with this, and the proof is complete.

\begin{remark}\sl Actually, the arguments of \cite{BHM1, BHM2} are a little bit more precise, and they lead to the lower bound (\cite{BHM2}, Theorem 1.2),
 $$\|u_{\eps}\|_{L^2([r_0,L]\times D_{\eps})}\geq \frac1{C}\eps^{1+\frac{n}2}e^{-\alpha_0r_0/\eps},$$
 in contradiction with (\ref{troppetit}).
\end{remark}

\end{document}